\documentclass{amsart}%
\usepackage{amsfonts}
\usepackage{amsmath}
\usepackage{amssymb}
\usepackage{graphicx}
\usepackage[latin1]{inputenc}
\usepackage[english]{babel}%
\newtheorem{theorem}{Theorem}[section]
\theoremstyle{plain}

\newtheorem{corollary}[theorem]{Corollary}

\newtheorem{lemma}[theorem]{Lemma}

\newtheorem{proposition}[theorem]{Proposition}

\numberwithin{equation}{section}
\begin{document}
\title[Parabolic-Type Equations With Variable Coefficients]{P-adic Elliptic Quadratic Forms, Parabolic-Type Pseudodifferential Equations
With Variable Coefficients, and Markov Processes.}
\author{O. F. Casas-Sánchez}
\address{Universidad Nacional de Colombia, Departamento de Matemáticas \\
Ciudad Universitaria, Bogotá D.C., Colombia.}
\email{ofcasass@unal.edu.co}
\author{W. A. Zúñiga-Galindo}
\address{Centro de Investigación y de Estudios Avanzados del Instituto Politécnico Nacional\\
Departamento de Matemáticas, Unidad Querétaro\\
Libramiento Norponiente \#2000, Fracc. Real de Juriquilla. Santiago de
Querétaro, Qro. 76230\\
México.}
\email{wazuniga@math.cinvestav.edu.mx}
\thanks{The second author was partially supported by CONACYT under Grant \# 127794.}
\subjclass[2000]{Primary 35S99, 47S10, 60J25; Secondary 11S85}
\keywords{Pseudodifferential operators, Markov processes, non-Archimedean analysis,
p-adic fields.}

\begin{abstract}
In this article we study the Cauchy problem for a new class of parabolic-type
pseudodifferential equations with variable coefficients for which the
fundamental solutions are transition density functions of Markov processes in
the four dimensional vector space over the field of $p$-adic numbers.

\end{abstract}
\maketitle

\section{Introduction}

The stochastic processes over the $p$-adics, or more generally over
ultrametric spaces, have attracted a lot of attention during the last thirty
years due to their connections with models of complex systems, such as glasses
and proteins, see e.g. \cite{Av-2}, \cite{Av-3}, \cite{Av-4}, \cite{Ch-Zuniga}%
, \cite{Ka}, \cite{K-Kos}, \cite{M-P-V}, \cite{R-T-V}, \cite{Var},
\cite{V-V-Z}, and the references therein. In particular, the Markov stochastic
process over the $p$-adics whose transition density functions are fundamental
solutions of pseudodifferential equations of parabolic-type have appeared in
several new models of complex systems. In \cite[Chapters 4, 5]{Koch} A. N.
Kochubei presented a general theory for one-dimensional parabolic-type
pseudodifferential equations with variable coefficients, whose fundamental
solutions are transition density functions for Markov processes in the
$p$-adic line. Taking into account the physical motivations before mentioned,
it is natural to try to develop a general $n$-dimensional theory for
pseudodifferential equations of parabolic-type over the $p$-adics. In this
article using the results of \cite{C-Zu1}-\cite[Chapter 4]{Koch} we present
four-dimensional analogs of the pseudodifferential equations of parabolic-type
with variable coefficients studied by A. N. Kochubei in \cite{Koch}, see also
\cite{Koch0}. We should mention that other types of $n$-dimensional
pseudodifferential equations have bee studied in \cite{Var}, \cite{Z-G3},
\cite{R-Zu}, \cite{G-Zuniga}, \cite{T-Zuniga}, \cite{Ch-Zuniga}, among others.
To explain the novelty of our contribution, we begin by saying that the theory
developed in \cite{Koch} depends heavily on an explicit formula for the
Fourier transform of the Riesz kernel associated to the symbol of the
Vladimirov pseudodifferential operator. It turns out, that this formula is a
particular case of the functional equation satisfied for certain local zeta
functions (i.e.\ distributions in an arithmetical framework) attached to
quadratic forms. This functional equation was established by S. Rallis and G.
Schiffmann in \cite{R-S}, see also the references cited in \cite{C-Zu1}. In
\cite{C-Zu1} the authors studied the Riesz kernels and pseudodifferential
operators associated to elliptic quadratic forms of dimensions two and four.
Among several results, we determined explicit formulas for the Fourier
transform of the Riesz kernels. In this article we use these results jointly
with classical techniques for parabolic equations to study the Cauchy problem
for a new class of parabolic-type pseudodifferential equations for which the
fundamental solutions are transition density functions of Markov processes in
the four dimensional vector space over the field of $p$-adic numbers.

The article is organized as follows. In Section \ref{Section1}, we fix the
notation and collect some results about Riesz kernels associated to elliptic
quadratic forms of dimension four. In Section \ref{Sect.Heat.Kernel} we
introduce the heat kernels and \ certain pseudodifferential operators attached
to elliptic quadratic forms and show some basic results needed for the other
sections. In Section \ref{Sect.Cauchy.Problem} we study the Cauchy problem for
pseudodifferential equations with constant coefficients, see Theorem
\ref{Thm1}. In Section \ref{Sect.Eq.Var.Coeff.} we study the existence and
uniqueness of the Cauchy problem for pseudodifferential equations with
variable coefficients, see Theorems \ref{Thm2}, \ref{Thm3}, we also discuss
the probabilistic meaning of the fundamental solutions, see Theorem \ref{Thm4}.

\section{\label{Section1}Preliminaries}

In this section we fix the notation and collect some basic results on $p$-adic
analysis that we will use through the article. For a detailed exposition on
$p$-adic analysis the reader may consult \cite{A-K-S}, \cite{Taibleson},
\cite{V-V-Z}.

\subsection{The field of $p$-adic numbers}

Along this article $p$ will denote a prime number different from $2$. The
field of $p-$adic numbers $\mathbb{Q}_{p}$ is defined as the completion of the
field of rational numbers $\mathbb{Q}$ with respect to the $p-$adic norm
$|\cdot|_{p}$, which is defined as
\[
|x|_{p}=%
\begin{cases}
0 & \text{if }x=0\\
p^{-\gamma} & \text{if }x=p^{\gamma}\dfrac{a}{b},
\end{cases}
\]
where $a$ and $b$ are integers coprime with $p$. The integer $\gamma:=ord(x)$,
with $ord(0):=+\infty$, is called the\textit{ }$p-$\textit{adic order of} $x$.
We extend the $p-$adic norm to $\mathbb{Q}_{p}^{n}$ by taking%
\[
||x||_{p}:=\max_{1\leq i\leq n}|x_{i}|_{p},\qquad\text{for }x=(x_{1}%
,\dots,x_{n})\in\mathbb{Q}_{p}^{n}.
\]
We define $ord(x)=\min_{1\leq i\leq n}\{ord(x_{i})\}$, then $||x||_{p}%
=p^{-\text{ord}(x)}$. Any $p-$adic number $x\neq0$ has a unique expansion
$x=p^{ord(x)}\sum_{j=0}^{\infty}x_{j}p^{j}$, where $x_{j}\in\{0,1,2,\dots
,p-1\}$ and $x_{0}\neq0$. By using this expansion, we define \textit{the
fractional part of }$x\in\mathbb{Q}_{p}$, denoted $\{x\}_{p}$, as the rational
number
\[
\{x\}_{p}=%
\begin{cases}
0 & \text{if }x=0\text{ or }ord(x)\geq0\\
p^{\text{ord}(x)}\sum_{j=0}^{-ord(x)-1}x_{j}p^{j} & \text{if }ord(x)<0.
\end{cases}
\]
For $\gamma\in\mathbb{Z}$, denote by $B_{\gamma}^{n}(a)=\{x\in\mathbb{Q}%
_{p}^{n}:||x-a||_{p}\leq p^{\gamma}\}$ \textit{the ball of radius }$p^{\gamma
}$ \textit{with center at} $a=(a_{1},\dots,a_{n})\in\mathbb{Q}_{p}^{n}$, and
take $B_{\gamma}^{n}(0):=B_{\gamma}^{n}$. Note that $B_{\gamma}^{n}%
(a)=B_{\gamma}(a_{1})\times\cdots\times B_{\gamma}(a_{n})$, where $B_{\gamma
}(a_{i}):=\{x\in\mathbb{Q}_{p}:|x-a_{i}|_{p}\leq p^{\gamma}\}$ is the
one-dimensional ball of radius $p^{\gamma}$ with center at $a_{i}\in
\mathbb{Q}_{p}$. The ball $B_{0}^{n}(0)$ is equals the product of $n$ copies
of $B_{0}(0):=\mathbb{Z}_{p}$, \textit{the ring of }$p-$\textit{adic
integers}. We denote by $\Omega(\left\Vert x\right\Vert _{p})$ the
characteristic function of $B_{0}^{n}(0)$. For more general sets, say Borel
sets, we use ${\LARGE 1}_{A}\left(  x\right)  $ to denote the characteristic
function of $A$.

\subsection{The Bruhat-Schwartz space and Fourier transform}

A complex-valued function $\varphi$ defined on $\mathbb{Q}_{p}^{n}$ is
\textit{called locally constant} if for any $x\in\mathbb{Q}_{p}^{n}$ there
exist an integer $l(x)$ such that%
\begin{equation}
\varphi(x+x^{\prime})=\varphi(x)\text{ for }x^{\prime}\in B_{l(x)}^{n}.
\label{local_constancy}%
\end{equation}
A function $\varphi:\mathbb{Q}_{p}^{n}\rightarrow\mathbb{C}$ is called a
\textit{Bruhat-Schwartz function (or a test function)} if it is locally
constant with compact support. The $\mathbb{C}$-vector space of
Bruhat-Schwartz functions is denoted by $S(\mathbb{Q}_{p}^{n}):=S$. For
$\varphi\in S(\mathbb{Q}_{p}^{n})$, the largest of such number $l=l(\varphi)$
satisfying (\ref{local_constancy}) is called \textit{the exponent of local
constancy of} $\varphi$.

Let $S^{\prime}(\mathbb{Q}_{p}^{n}):=S^{\prime}$ denote the set of all
functionals (distributions) on $S(\mathbb{Q}_{p}^{n})$. All functionals on
$S(\mathbb{Q}_{p}^{n})$ are continuous.

Set $\chi(y)=\exp(2\pi i\{y\}_{p})$ for $y\in\mathbb{Q}_{p}$. The map
$\chi(\cdot)$ is an additive character on $\mathbb{Q}_{p}$, i.e. a continuos
map from $\mathbb{Q}_{p}$ into the unit circle satisfying $\chi(y_{0}%
+y_{1})=\chi(y_{0})\chi(y_{1})$, $y_{0},y_{1}\in\mathbb{Q}_{p}$.

Given $\xi=(\xi_{1},\dots,\xi_{n})$ and $x=(x_{1},\dots,x_{n})\in
\mathbb{Q}_{p}^{n}$, we set $\xi\cdot x:=\sum_{j=1}^{n}\xi_{j}x_{j}$. The
Fourier transform of $\varphi\in S(\mathbb{Q}_{p}^{n})$ is defined as
\[
(\mathcal{F}\varphi)(\xi)=\int_{\mathbb{Q}_{p}^{n}}\chi(-\xi\cdot
x)\varphi(x)d^{n}x\quad\text{for }\xi\in\mathbb{Q}_{p}^{n},
\]
where $d^{n}x$ is the Haar measure on $\mathbb{Q}_{p}^{n}$ normalized by the
condition $vol(B_{0}^{n})=1$. The Fourier transform is a linear isomorphism
from $S(\mathbb{Q}_{p}^{n})$ onto itself satisfying $(\mathcal{F}%
(\mathcal{F}\varphi))(\xi)=\varphi(-\xi)$. We will also use the notation
$\mathcal{F}_{x\rightarrow\xi}\varphi$ and $\widehat{\varphi}$\ for the
Fourier transform of $\varphi$.

The Fourier transform $\mathcal{F}\left[  f\right]  $ of a distribution $f\in
S^{\prime}\left(  \mathbb{Q}_{p}^{n}\right)  $ is defined by%
\[
\left(  \mathcal{F}\left[  f\right]  ,\varphi\right)  =\left(  f,\mathcal{F}%
\left[  \varphi\right]  \right)  \text{ for all }\varphi\in S\left(
\mathbb{Q}_{p}^{n}\right)  \text{.}%
\]
The Fourier transform $f\rightarrow\mathcal{F}\left[  f\right]  $ is a linear
isomorphism from $S^{\prime}\left(  \mathbb{Q}_{p}^{n}\right)  $\ onto
$S^{\prime}\left(  \mathbb{Q}_{p}^{n}\right)  $. Furthermore, $f=\mathcal{F}%
\left[  \mathcal{F}\left[  f\right]  \left(  -\xi\right)  \right]  $.

\subsection{Elliptic Quadratic Forms and Riesz Kernels}

We set
\[
f(x)=x_{1}^{2}-ax_{2}^{2}-px_{3}^{2}+apx_{4}^{2}\text{, \ }f^{\circ}\left(
\xi\right)  =ap\xi_{1}^{2}-p\xi_{2}^{2}-a\xi_{3}^{2}+\xi_{4}^{2}%
\]
with $a\in\mathbb{Z}$ a quadratic non-residue module $p$. The form $f(x)$ is
elliptic, i.e. $f(x)=0\Leftrightarrow x=0$, and $f^{\circ}\left(  \xi\right)
=ap{f\left(  \frac{\xi_{1}}{1},\frac{\xi_{2}}{-a},\frac{\xi_{3}}{-p},\frac
{\xi_{4}}{ap}\right)  }$, hence $f^{\circ}\left(  \xi\right)  $ is also
elliptic. The following estimate will be used frequently along the article:
\ there exist positive constants $A$, $B$ such that%

\begin{equation}
B\left\Vert x\right\Vert _{p}^{2}\leq|f\left(  x\right)  |_{p}\leq A\left\Vert
x\right\Vert _{p}^{2}\text{ for any }x\in\mathbb{Q}_{p}^{n},
\label{Zuniga_ineq}%
\end{equation}
c.f. \cite[Lemma 1]{Z-G3}.

The Riesz kernels attached to $f$\ and $f^{\circ}$ satisfy the following
functional equation:
\begin{equation}
\int\limits_{\mathbb{Q}_{p}^{4}\setminus\{0\}}|f(z)|^{s-2}\widehat{\varphi
}(z)d^{4}z=\frac{1-p^{s-2}}{1-p^{-s}}\int\limits_{\mathbb{Q}_{p}^{4}%
\setminus\{0\}}|f^{\circ}\left(  \xi\right)  |_{p}^{-s}\varphi(\xi)d^{4}\xi,
\label{Functional_eq.}%
\end{equation}
for $\varphi\in S\left(  \mathbb{Q}_{p}^{4}\right)  $ and $s\in\mathbb{C}$,
c.f. \cite[Proposition 2.8]{C-Zu1}. For further details about Riesz kernels
attached to quadratic forms we refer the reader to \cite{C-Zu1}.

\begin{lemma}
\label{lemma_0}If $\alpha>0$, then
\[
|f^{\circ}(x)|_{p}^{\alpha}=\frac{1-p^{\alpha}}{1-p^{-\alpha-2}}%
\int\limits_{\mathbb{Q}_{p}^{4}}|f(\xi)|_{p}^{-\alpha-2}[\chi(\xi\cdot
x)-1]d^{4}\xi.
\]

\end{lemma}

\begin{proof}
The formula follows from (\ref{Functional_eq.}) by using the argument given in
\cite{Koch} for Proposition 2.3.
\end{proof}

\section{\label{Sect.Heat.Kernel}Heat Kernels}

We define \textit{the heat kernel} attached to $f^{\circ}$ as%
\[
Z(x,t):=Z(x,t;f^{\circ},\alpha,\kappa)=%
{\displaystyle\int\limits_{\mathbb{Q}_{p}^{4}}}
\chi(\xi\cdot x)e^{-\kappa t|f^{\circ}(\xi)|_{p}^{\alpha}}d^{4}\xi
\]
for \ $x\in\mathbb{Q}_{p}^{4}$, $t>0$, $\alpha>0$, and $\kappa>0$. When
considering $Z(x,t)$ as a function of $x$ for $t$ fixed we will write
$Z_{t}(x)$. The heat kernels considered in this article are a particular case
of those studied in \cite{Z-G3}.

Given $M\in\mathbb{Z}$, we set
\[
Z_{t}^{(M)}(x):=%
{\displaystyle\int\limits_{\mathbb{Q}_{p}^{4}}}
\Omega(p^{-M}||\xi||_{p})\chi(\xi\cdot x)e^{-\kappa t|f^{\circ}(\xi
)|_{p}^{\alpha}}d^{4}\xi
\]
for $x\in\mathbb{Q}_{p}^{4}$, $t>0$, $\alpha>0$, and $\kappa>0$. Taking into
account that
\[
\left\vert \Omega(p^{-M}||\xi||_{p})\chi(\xi\cdot x)e^{-\kappa t|f^{\circ}%
(\xi)|_{p}^{\alpha}}\right\vert \leq e^{-\kappa t|f^{\circ}(\xi)|_{p}^{\alpha
}}\leq e^{-\kappa tB^{\alpha}||\xi||_{p}^{2\alpha}}\in L^{1}(\mathbb{Q}%
_{p}^{4}),
\]
c.f. (\ref{Zuniga_ineq}), the Dominated Convergence Theorem implies that
\[
\lim_{M\rightarrow\infty}Z_{t}^{(M)}(x)=Z_{t}(x)\text{ for }x\in\mathbb{Q}%
_{p}^{4}\text{ and for }t>0\text{.}%
\]

\begin{proposition}
\label{Prop1} For $x,\xi\in\mathbb{Q}_{p}^{4}\setminus\{0\}$ the following
assertions hold:

\noindent(i) ${|f(x+\xi)|_{p}=|f(x)|_{p}}$, for $||\xi||_{p}<p^{-1}||x||_{p}$;

\noindent(ii) ${Z(x,t)=\sum_{m=1}^{\infty}\frac{(-1)^{m}}{m!}\left(
\frac{1-p^{\alpha m}}{1-p^{-\alpha m-2}}\right)  \kappa^{m}t^{m}%
|f(x)|_{p}^{-\alpha m-2}}${;}

\noindent(iii) ${Z(x+\xi,t)=Z(x,t)}${,}${\text{ for }||\xi||}_{p}%
{<p^{-1}||x||}_{p}${;}

\noindent(iv) ${Z(x,t)\geq0}${,}${\text{ for }x\in\mathbb{Q}_{p}^{4}\text{ and
}t>0}${.}
\end{proposition}

\begin{proof}
(i) Set $x=p^{ord(x)}u$,$\ \xi=p^{ord(\xi)}v$ with $||u||_{p}=||v||_{p}=1$.
Then
\begin{align*}
|f(x+\xi)|_{p}  &  =|f(p^{ord(x)}u+p^{ord(\xi)}v)|_{p}=|p^{2ord(x)}%
f(u+p^{ord(\xi)-ord(x)}v)|_{p}\\
&  =p^{-2ord(x)}|f(u)+p^{ord(\xi)-ord(x)}A|_{p}%
\end{align*}
for some $A\in\mathbb{Z}_{p}$. Note that $|f(u)|_{p}\in\{1,p^{-1}\}$, then for
$ord(\xi)-ord(x)>1$, $|f(u)+p^{ord(\xi)-ord(x)}A|_{p}=|f(u)|_{p}$ and
\[
|f(x+\xi)|_{p}=p^{-2ord(x)}|f(u)|_{p}=|p^{2ord(x)}f(u)|_{p}=|f(x)|_{p}.
\]

(ii) By using the Taylor expansion of $e^{x}$ and Fubini's Theorem,
$Z_{t}^{(M)}(x)$ can be rewritten as
\[
Z_{t}^{(M)}(x)=\sum_{m=0}^{\infty}\frac{(-1)^{m}}{m!}\kappa^{m}t^{m}%
{\displaystyle\int\limits_{\mathbb{Q}_{p}^{4}}}
\Omega(p^{-M}||\xi||_{p})\chi(\xi\cdot x)|f^{\circ}(\xi)|_{p}^{m\alpha}%
d^{4}\xi.
\]
By using (\ref{Functional_eq.}) with $m\neq0$, we have
\begin{align*}
I_{M}  &  :=\int_{\mathbb{Q}_{p}^{4}}\Omega(p^{-M}||\xi||_{p})\chi(\xi\cdot
x)|f^{\circ}(\xi)|_{p}^{m\alpha}d^{4}\xi\\
&  =\frac{1-p^{\alpha m}}{1-p^{-\alpha m-2}}%
{\displaystyle\int\limits_{\mathbb{Q}_{p}^{4}}}
\mathcal{F}_{\xi\rightarrow z}\left[  \Omega(p^{-M}||\xi||_{p})\chi(\xi\cdot
x)\right]  |f(z)|_{p}^{-m\alpha-2}d^{4}z\\
&  =\frac{1-p^{\alpha m}}{1-p^{-\alpha m-2}}p^{4M}%
{\displaystyle\int\limits_{\mathbb{Q}_{p}^{4}}}
\Omega(p^{M}||x-z||_{p})|f(z)|_{p}^{-m\alpha-2}d^{4}z.
\end{align*}
We now use the fact that $x\neq0$ is fixed, and take $M>1+ord(x)$, changing
variables as $z=x-p^{M}y$ in $I_{M}$, and using (i), we have
\begin{align*}
I_{M}  &  =\frac{1-p^{\alpha m}}{1-p^{-\alpha m-2}}%
{\displaystyle\int\limits_{\mathbb{Q}_{p}^{4}}}
\Omega(||y||_{p})|f(x-p^{M}y)|_{p}^{-m\alpha-2}d^{4}y\\
&  =\frac{1-p^{\alpha m}}{1-p^{-\alpha m-2}}|f(x)|_{p}^{-m\alpha-2}%
{\displaystyle\int\limits_{\mathbb{Q}_{p}^{4}}}
\Omega(||y||_{p})d^{4}y=\frac{1-p^{\alpha m}}{1-p^{-\alpha m-2}}%
|f(x)|_{p}^{-m\alpha-2}.
\end{align*}
On the \ other hand, for $m=0$,
\[
\lim_{M\rightarrow\infty}%
{\displaystyle\int\limits_{\mathbb{Q}_{p}^{4}}}
\Omega(p^{-M}||\xi||_{p})\chi(\xi\cdot x)d^{4}\xi=0.
\]
Therefore%
\[
Z(x,t)=\sum_{m=1}^{\infty}\frac{(-1)^{m}}{m!}\kappa^{m}t^{m}\left(
\frac{1-p^{\alpha m}}{1-p^{-\alpha m-2}}\right)  |f(x)|_{p}^{-m\alpha-2}\text{
for }x\in\mathbb{Q}_{p}^{4}\setminus\{0\}\text{.}%
\]

(iii) It is consequence of (ii) and (i).

(iv) See \cite[Theorem 2]{Z-G3}.
\end{proof}

\begin{proposition}
\label{Prop2} The following assertions hold for any $x\in\mathbb{Q}_{p}^{4}$,
$t>0$:

\noindent(i) there exists a positive constant $C_{1}$ such that $Z(x,t)\leq
C_{1}t(t^{1/2\alpha}+||x||_{p})^{-2\alpha-4}$;\ 
\end{proposition}

\noindent(ii) $Z(\cdot,t)\in C^{1}\left(  \left(  0,\infty\right)
,\mathbb{R}\right)  $ and ${\dfrac{\partial Z(x,t)}{\partial t}=-}${$\kappa$%
}$\int_{\mathbb{Q}_{p}^{4}}{|}${$f^{\circ}$}${(\eta)|_{p}^{\alpha}e^{-\kappa
t|f^{\circ}(\eta)|^{\alpha}}\chi(x\cdot\eta)d}^{4}{\eta}$;

\noindent(iii) there exists a positive constant $C_{2}$ such that
\[
{\left\vert \dfrac{\partial Z(x,t)}{\partial t}\right\vert \leq C_{2}\left(
t^{1/2\alpha}+||x||_{p}\right)  ^{-2\alpha-4}}.
\]

\begin{proof}
(i) We first consider the case in which $t||x||_{p}^{-2\alpha}\leq1$, then by
Proposition \ref{Prop1} (ii)-(iv) and by (\ref{Zuniga_ineq}),
\begin{align*}
Z(x,t)  &  \leq|f(x)|_{p}^{-2}\sum_{m=1}^{\infty}\frac{C_{0}^{m}}%
{m!}(t|f(x)|_{p}^{-\alpha})^{m}\leq B^{-2}||x||_{p}^{-4}\sum_{m=1}^{\infty
}\frac{(C_{0}B^{-\alpha})^{m}}{m!}(t||x||_{p}^{-2\alpha})^{m}\\
&  =B^{-2}||x||_{p}^{-4}\left(  e^{C_{0}B^{-\alpha}t||x||_{p}^{-2\alpha}%
}-1\right)  \leq Ct||x||_{p}^{-2\alpha-4}.
\end{align*}
We now consider the case in which $t>0$. Take $k$ to be an integer satisfying
$p^{k-1}\leq t^{\frac{1}{2\alpha}}\leq p^{k}$. Then by Proposition \ref{Prop1}
(iv), and by (\ref{Zuniga_ineq}),
\begin{align*}
|Z(x,t)|  &  =Z(x,t)\leq%
{\displaystyle\int\limits_{\mathbb{Q}_{p}^{4}}}
e^{-\kappa t|f^{\circ}(\xi)|_{p}^{\alpha}}d^{4}\xi\leq%
{\displaystyle\int\limits_{\mathbb{Q}_{p}^{4}}}
e^{-C_{0}t|\left\vert \xi\right\vert |_{p}^{2\alpha}}d^{4}\xi\\
&  \leq%
{\displaystyle\int\limits_{\mathbb{Q}_{p}^{4}}}
e^{-C_{0}\left\Vert p^{1-k}\xi\right\Vert _{p}^{2\alpha}}d^{4}\xi\leq
p^{-4k-k}%
{\displaystyle\int\limits_{\mathbb{Q}_{p}^{4}}}
e^{-C_{0}||\eta||^{2\alpha}}d^{4}\eta\leq Ct^{-2/\alpha}.
\end{align*}
By combining the above inequalities, see for instance the end of the proof of
Proposition \ref{Prop4}, we get the announced result.

(ii) The formula for $\dfrac{\partial Z(x,t)}{\partial t}$ is obtained by a
straightforward calculation. The continuity of $Z(\cdot,t)$ is obtained from
the formula for $\dfrac{\partial Z(x,t)}{\partial t}$ by using the Dominated
Convergence Theorem. (iii) This part is proved in the same way as (i).
\end{proof}

The first part of Proposition \ref{Prop2} is a particular case of Theorem 1 in
\cite{Z-G3}. We include this proof here due to two reasons: first, it shows a
very deep connection between the functional equation (\ref{Functional_eq.})
and the heat kernels; second, we use this technique for bounding several types
of oscillatory integrals in this article.

\begin{corollary}
\label{cor2}(i) $\int_{\mathbb{Q}_{p}^{4}}Z_{t}(x)d^{4}x=1$ for $t>0$; (ii)
$Z_{t}(x)\in L^{\rho}$ for $t>0$ and for $1\leq\rho\leq\infty$.
\end{corollary}

\section{Some Results on Operators of type $\mathbf{f}(\partial,\alpha)$}

\subsection{The space $\mathfrak{M_{\lambda}}$}

We denote by $\mathfrak{M}_{\lambda}$, $\lambda\geq0$, the $\mathbb{C}$-vector
space of locally constant functions $\varphi(x)$ on $\mathbb{Q}_{p}^{4}$ such
that $|\varphi(x)|\leq C(1+||x||^{\lambda})$, where $C$ is a positive
constant. If the function $\varphi$ depends also on a parameter $t$, we shall
say that $\varphi\in\mathfrak{M}_{\lambda}$ \textit{uniformly with respect to}
$t$, if its constant $C$ and its exponent of local constancy do not depend on
$t$.

\begin{lemma}
\label{lemma3} If $\varphi\in\mathfrak{M}_{2\lambda}$, with $0\leq
\lambda<\alpha$ and $\alpha>0$, then
\[
\lim_{t\rightarrow0^{+}}%
{\displaystyle\int\limits_{\mathbb{Q}_{p}^{4}}}
Z(x-\xi,t)\varphi(\xi)d^{4}\xi=\varphi(x).
\]

\end{lemma}

\begin{proof}
By Corollary \ref{cor2} (i) and Proposition \ref{Prop2} (i), and the fact that
$\varphi$\ is locally constant,
\begin{multline*}
I:=\left\vert
{\displaystyle\int\limits_{\mathbb{Q}_{p}^{4}}}
Z(x-\xi,t)\varphi(\xi)d^{4}\xi-\varphi(x)\right\vert =\left\vert
{\displaystyle\int\limits_{\mathbb{Q}_{p}^{4}}}
Z(x-\xi,t)[\varphi(\xi)-\varphi(x)]d^{4}\xi\right\vert \\
=\left\vert \text{ }%
{\displaystyle\int\limits_{\left\Vert x-\xi\right\Vert _{p}\geq p^{L}}}
Z(x-\xi,t)[\varphi(\xi)-\varphi(x)]d^{4}\xi\right\vert \\
\leq C_{1}t%
{\displaystyle\int\limits_{\left\Vert x-\xi\right\Vert _{p}\geq p^{L}}}
(t^{1/2\alpha}+||x-\xi||_{p})^{-2\alpha-4}\left\vert \varphi(\xi
)-\varphi(x)\right\vert d^{4}\xi\\
=C_{1}t%
{\displaystyle\int\limits_{\left\Vert z\right\Vert _{p}\geq p^{L}}}
(t^{1/2\alpha}+||z||_{p})^{-2\alpha-4}\left\vert \varphi(x-z)-\varphi
(x)\right\vert d^{4}z.
\end{multline*}
By applying the triangle inequality in the last integral and noticing that%
\[
\left\vert \varphi(x)\right\vert
{\displaystyle\int\limits_{\left\Vert z\right\Vert _{p}\geq p^{L}}}
(t^{1/2\alpha}+||z||_{p})^{-2\alpha-4}d^{4}\xi\leq\left\vert \varphi
(x)\right\vert
{\displaystyle\int\limits_{\left\Vert z\right\Vert _{p}\geq p^{L}}}
||z||_{p}^{-2\alpha-4}d^{4}\xi\leq C_{0}\left\vert \varphi(x)\right\vert ,
\]
and
\[%
{\displaystyle\int\limits_{\left\Vert z\right\Vert _{p}\geq p^{L}}}
(t^{1/2\alpha}+||z||_{p})^{-2\alpha-4}||z||_{p}^{2\lambda}d^{4}\xi\leq%
{\displaystyle\int\limits_{\left\Vert z\right\Vert _{p}\geq p^{L}}}
||z||_{p}{}^{-2\alpha+2\lambda-4}d^{4}z<\infty,
\]
we have
\[
\lim_{t\rightarrow0^{+}}I\leq\left(  C_{1}+C_{2}\left\vert \varphi
(x)\right\vert \right)  \lim_{t\rightarrow0^{+}}t=0.
\]

\end{proof}

\subsection{The Operator $\mathbf{f}(\partial,\alpha)$}

Given $\alpha>0$, we define the pseudodifferential operator with symbol
$\left\vert f^{\circ}\left(  \xi\right)  \right\vert _{p}^{\alpha}$ by%
\[%
\begin{array}
[c]{lll}%
\mathbf{S}\left(  \mathbb{Q}_{p}^{4}\right)  & \rightarrow & C\left(
\mathbb{Q}_{p}^{4}\right)  \cap L^{2}\left(  \mathbb{Q}_{p}^{4}\right) \\
&  & \\
\varphi & \rightarrow & \left(  \boldsymbol{f}\left(  \partial,\alpha\right)
\varphi\right)  \left(  x\right)  =\mathcal{F}_{\xi\rightarrow x}^{-1}\left(
\left\vert f^{\circ}\left(  \xi\right)  \right\vert _{p}^{\alpha}%
\mathcal{F}_{x\rightarrow\xi}\varphi\right)  .
\end{array}
\]
This operator is well-defined since $\left\vert f^{\circ}\left(  \xi\right)
\right\vert _{p}^{\alpha}\mathcal{F}_{x\rightarrow\xi}\varphi\in L^{1}\left(
\mathbb{Q}_{p}^{4}\right)  \cap L^{2}\left(  \mathbb{Q}_{p}^{4}\right)  $. By
\cite[Proposition 3.4 (iv)]{C-Zu1},
\begin{equation}
\left(  \boldsymbol{f}\left(  \partial,\alpha\right)  \varphi\right)  \left(
x\right)  =\dfrac{1-p^{\alpha}}{1-p^{-\alpha-2}}%
{\displaystyle\int\nolimits_{\mathbb{Q}_{p}^{4}}}
\frac{\varphi(x-y)-\varphi(x)}{|f(y)|_{p}^{\alpha+2}}d^{4}y, \label{Ext_oper}%
\end{equation}
for $\varphi\in\mathbf{S}\left(  \mathbb{Q}_{p}^{4}\right)  $. The operator
$\boldsymbol{f}\left(  \partial,\alpha\right)  $\ can be extended to any
locally constant functions $\varphi\left(  x\right)  $ satisfying%
\begin{equation}%
{\displaystyle\int\limits_{\left\Vert x\right\Vert _{p}\geq p^{m}}}
\frac{\left\vert \varphi(x)\right\vert }{|f(x)|_{p}^{\alpha+2}}d^{4}%
x<\infty\text{ for some }m\in\mathbb{Z}\text{,} \label{condicion1}%
\end{equation}
c.f. \cite[Lemma 4.1]{C-Zu1}.

Note that
\[
{Z_{t}^{(M)}(x)=}%
{\displaystyle\int\limits_{||\eta||_{p}\leq p^{M}}}
{\chi(x\cdot\eta)e^{-\kappa t|f^{\circ}(\eta)|_{p}^{\alpha}}d}^{4}{\eta
}\text{, with }M\in\mathbb{N}%
\]
is a locally constant and bounded function, c.f. Proposition \ref{Prop2}.
Furthermore, by Proposition \ref{Prop2} and (\ref{Zuniga_ineq}), ${Z_{t}%
^{(M)}(x)}$ satisfies condition (\ref{condicion1}), for $t>0$.

\begin{lemma}
\label{lemma_1}%
\begin{equation}
(\mathbf{f}(\partial,\gamma)Z_{t}^{(M)})(x)=%
{\displaystyle\int\limits_{||\eta||_{p}\leq p^{M}}}
\chi(x\cdot\eta)|f^{\circ}(\eta)|_{p}^{\gamma}e^{-at|f^{\circ}(\eta
)|_{p}^{\alpha}}d^{4}\eta\text{,} \label{f_M}%
\end{equation}
for $M\in\mathbb{N}$ and for $t>0$.
\end{lemma}

\begin{proof}
Note that if $||\xi||_{p}\leq p^{-M}$, then $Z_{t}^{(M)}(x-\xi)=Z_{t}%
^{(M)}(x)$. In addition, since ${Z_{t}^{(M)}(x)}$ satisfies condition
(\ref{condicion1}), we can use formula (\ref{Ext_oper}) to compute
$(\mathbf{f}(\partial,\gamma)Z_{t}^{(M)})(x)$ as follows:%
\begin{multline*}
(\mathbf{f}(\partial,\gamma)Z_{t}^{(M)})(x)=\frac{1-p^{\gamma}}{1-p^{-\gamma
-2}}%
{\displaystyle\int\limits_{\mathbb{Q}_{p}^{4}}}
|f(\xi)|_{p}^{-\gamma-2}\left[  Z_{t}^{(M)}(x-\xi)-Z_{t}^{(M)}(x)\right]
d^{4}\xi\\
=\frac{1-p^{\gamma}}{1-p^{-\gamma-2}}%
{\displaystyle\int\limits_{||\xi||_{p}>p^{-M}}}
|f(\xi)|_{p}^{-\gamma-2}\left[  Z_{t}^{(M)}(x-\xi)-Z_{t}^{(M)}(x)\right]
d^{4}\xi\\
+\frac{1-p^{\gamma}}{1-p^{-\gamma-2}}%
{\displaystyle\int\limits_{||\xi||_{p}\leq p^{-M}}}
|f(\xi)|_{p}^{-\gamma-2}\left[  Z_{t}^{(M)}(x-\xi)-Z_{t}^{(M)}(x)\right]
d^{4}\xi\\
=\frac{1-p^{\gamma}}{1-p^{-\gamma-2}}%
{\displaystyle\int\limits_{||\xi||_{p}>p^{-M}}}
|f(\xi)|_{p}^{-\gamma-2}\left[  Z_{t}^{(M)}(x-\xi)-Z_{t}^{(M)}(x)\right]
d^{4}\xi\\
=\frac{1-p^{\gamma}}{1-p^{-\gamma-2}}%
{\displaystyle\int\limits_{||\xi||_{p}>p^{-M}}}
|f(\xi)|_{p}^{-\gamma-2}%
{\displaystyle\int\limits_{||\eta||_{p}\leq p^{M}}}
e^{-at|f^{\circ}(\eta)|_{p}^{\alpha}}\chi(x\cdot\eta)[\chi(\xi\cdot
\eta)-1]d^{4}\eta d^{4}\xi\\
=\frac{1-p^{\gamma}}{1-p^{-\gamma-2}}%
{\displaystyle\int\limits_{||\eta||_{p}\leq p^{M}}}
e^{-at|f^{\circ}(\eta)|_{p}^{\alpha}}\chi(x\cdot\eta)%
{\displaystyle\int\limits_{||\xi||_{p}>p^{-M}}}
|f(\xi)|_{p}^{-\gamma-2}[\chi(\xi\cdot\eta)-1]d^{4}\xi d^{4}\eta\\
=\frac{1-p^{\gamma}}{1-p^{-\gamma-2}}%
{\displaystyle\int\limits_{||\eta||_{p}\leq p^{M}}}
e^{-at|f^{\circ}(\eta)|_{p}^{\alpha}}\chi(x\cdot\eta)\left\{
{\displaystyle\int\limits_{\mathbb{Q}_{p}^{4}}}
|f(\xi)|_{p}^{-\gamma-2}[\chi(\xi\cdot\eta)-1]d^{4}\xi\right\}  d^{4}\eta\\
=%
{\displaystyle\int\limits_{||\eta||_{p}\leq p^{M}}}
|f^{\circ}(\eta)|_{p}^{\gamma}e^{-at|f^{\circ}(\eta)|_{p}^{\alpha}}\chi
(x\cdot\eta)d^{4}\eta\text{, c.f. Lemma \ref{lemma_0}.}%
\end{multline*}

\end{proof}

By Proposition \ref{Prop1} (iii) and Proposition \ref{Prop2} (i),
$(\mathbf{f}(\partial,\gamma)Z_{t})(x)$ is well-defined for $x\neq0$ and for
$0<\gamma\leq\alpha$.

\begin{proposition}
\label{Prop3}%
\begin{equation}
(\mathbf{f}(\partial,\gamma)Z_{t})(x)=%
{\displaystyle\int\limits_{\mathbb{Q}_{p}^{4}}}
|f^{\circ}(\eta)|_{p}^{\gamma}e^{-\kappa t|f^{\circ}(\eta)|^{\alpha}}%
\chi(x\cdot\eta)d^{4}\eta,\text{\ for }0<\gamma\leq\alpha\text{, }t>0.
\label{formula_op_f}%
\end{equation}

\end{proposition}

\begin{proof}
By (\ref{Zuniga_ineq}), $|f^{\circ}(\cdot)|_{p}^{\gamma}e^{-\kappa t|f^{\circ
}(\cdot)|_{p}^{\alpha}}\in L^{1}\left(  \mathbb{Q}_{p}^{4}\right)  $ for
$t>0$, then from (\ref{f_M}), by the Dominated Convergence Theorem, we obtain%
\begin{equation}
\lim_{M\rightarrow\infty}(\mathbf{f}(\partial,\gamma)Z^{(M)})(x,t)=%
{\displaystyle\int\limits_{\mathbb{Q}_{p}^{4}}}
\chi(x\cdot\eta)|f^{\circ}(\eta)|_{p}^{\gamma}e^{-\kappa t|f^{\circ}%
(\eta)|_{p}^{\alpha}}d^{4}\eta\text{, for }t>0. \label{formula}%
\end{equation}

On the \ other hand, fixing an $x\neq0$, by Proposition \ref{Prop1} (i),
\[
(\mathbf{f}(\partial,\gamma)Z_{t}^{(M)})(x)=\frac{1-p^{\gamma}}{1-p^{-\gamma
-2}}%
{\displaystyle\int\limits_{||\xi||_{p}>p^{-1}||x||_{p}}}
|f(\xi)|_{p}^{-\gamma-2}\left[  Z_{t}^{(M)}(x-\xi)-Z_{t}^{(M)}(x)\right]
d^{4}\xi.
\]
Finally, by the Dominated Convergence Theorem and (\ref{formula}), we have%
\[
\lim_{M\rightarrow\infty}(\mathbf{f}(\partial,\gamma)Z_{t}^{(M)}%
)(x)=(\mathbf{f}(\partial,\gamma)Z)(x,t)=%
{\displaystyle\int\limits_{\mathbb{Q}_{p}^{4}}}
\chi(x\cdot\eta)|f^{\circ}(\eta)|_{p}^{\gamma}e^{-\kappa t|f^{\circ}%
(\eta)|_{p}^{\alpha}}d^{4}\eta.
\]

Finally, we note that the right-hand side of (\ref{formula_op_f}) is
continuous at $x=0$.
\end{proof}

\begin{corollary}
\label{cor3}$\frac{\partial Z\left(  x,t\right)  }{\partial t}=-\kappa
(\mathbf{f}(\partial,\alpha)Z)(x,t)$ for $t>0$.
\end{corollary}

\begin{proof}
The formula follows from Propositions \ref{Prop3} and \ref{Prop2} (ii).
\end{proof}

\begin{proposition}
\label{Prop4} If $0<\gamma\leq\alpha$, then
\[
|(\mathbf{f}(\partial,\gamma)Z_{t})(x)|\leq C(t^{1/2\alpha}+||x||_{p}%
)^{-2\gamma-4}\text{, for }x\in\mathbb{Q}_{p}^{4}\text{ and for }t>0.
\]

\end{proposition}

\begin{proof}
By reasoning as in the proof of Proposition \ref{Prop1} (ii), we have%
\begin{equation}
(\mathbf{f}(\partial,\gamma)Z_{t})(x)=\sum_{m=1}^{\infty}\frac{(-1)^{m}}%
{m!}\kappa^{m}t^{m}\left(  \frac{1-p^{\alpha m+\gamma}}{1-p^{-\alpha
m-\gamma-2}}\right)  |f(x)|_{p}^{-m\alpha-\gamma-2}. \label{formula_2}%
\end{equation}
If $t||x||^{-2\alpha}\leq1$, from (\ref{formula_2}) and (\ref{Zuniga_ineq}),
we obtain
\begin{equation}
|(\mathbf{f}(\partial,\gamma)Z_{t})(x)|\leq|f(x)|_{p}^{-\gamma-2}\sum
_{m=1}^{\infty}\frac{C^{m}}{m!}(t|f(x)|_{p}^{-\alpha})^{m}\leq C_{1}%
||x||_{p}^{-2\gamma-4}. \label{Enq_1}%
\end{equation}
On other hand, take $k$ such that $p^{k-1}\leq t^{1/2\alpha}\leq p^{k}$. From
(\ref{formula_op_f}) by using (\ref{Zuniga_ineq}), we get
\begin{align}
|(\mathbf{f}(\partial,\gamma)Z_{t})(x)|  &  \leq A^{\gamma}%
{\displaystyle\int\limits_{\mathbb{Q}_{p}^{4}}}
||\eta||_{p}^{2\gamma}e^{-\kappa tB^{\alpha}||\eta||_{p}^{2\alpha}}d\eta\leq
A^{\gamma}%
{\displaystyle\int\limits_{\mathbb{Q}_{p}^{4}}}
||\eta||_{p}^{2\gamma}e^{-aB^{\alpha}||p^{-(k-1)}\eta||_{p}^{2\alpha}}%
d\eta\nonumber\\
&  =A^{\gamma}p^{-4(k-1)-2\gamma(k-1)}%
{\displaystyle\int\limits_{\mathbb{Q}_{p}^{4}}}
||\xi||_{p}^{2\gamma}e^{-aB^{\alpha}||\xi||_{p}^{2\alpha}}d^{4}\xi\leq
Ct^{-4-2\gamma/2\alpha}. \label{Enq_2}%
\end{align}
The announced results follows from inequalities\ (\ref{Enq_1})-(\ref{Enq_2}).
Indeed, $t||x||_{p}^{-2\alpha}\leq1$ implies that $||x||_{p}\geq
\frac{||x||_{p}}{2}+\frac{t^{1/2\alpha}}{2}$, and hence
\[
||x||_{p}^{-2\gamma-4}\leq2^{2\gamma+4}\left(  ||x||_{p}+t^{1/2\alpha}\right)
^{-2\gamma-4}.
\]
Now, if $t||x||^{-2\alpha}>1$, then $t^{1/2\alpha}>\frac{t^{1/2\alpha}}%
{2}+\frac{||x||_{p}}{2}$ and
\[
t^{-4-2\gamma/2\alpha}<2^{2\gamma+4}\left(  t^{1/2\alpha}+||x||_{p}\right)
^{-2\gamma-4}.
\]

\end{proof}

\begin{corollary}
\label{cor4}%
\[%
{\displaystyle\int\limits_{\mathbb{Q}_{p}^{4}}}
(\mathbf{f}(\partial,\gamma)Z_{t})(x)d^{4}x=0\text{ for }t>0.
\]

\end{corollary}

\section{\label{Sect.Cauchy.Problem}The Cauchy problem}

Along this section, we fix the domain ($Dom\left(  \mathbf{f}\right)  $) of
the operator $\mathbf{f}(\partial,\alpha)$ to be the $\mathbb{C}$-vector space
of locally constant functions satisfying (\ref{condicion1}), and
$\mathbf{f}(\partial,\alpha)\varphi$ is given by (\ref{Ext_oper}) for
$\varphi\in Dom\left(  \mathbf{f}\right)  $. Note that $\mathfrak{M}%
_{2\lambda}\subset Dom\left(  \mathbf{f}\right)  $ for $\lambda<\alpha$.

In this section we study the following Cauchy problem:%

\begin{equation}
\left\{
\begin{array}
[c]{ll}%
\frac{\partial u(x,t)}{\partial t}+\kappa\mathbf{f}(\partial,\gamma
)u(x,t)=g(x,t), & x\in\mathbb{Q}_{p}^{4},\quad0<t\leq T,\\
& \\
u(x,0)=\varphi(x) &
\end{array}
\right.  \label{Cauchy_Pr_1}%
\end{equation}
where $\kappa>0$, $\alpha>0$, $T>0$, $\varphi\in\mathfrak{M}_{2\lambda}$,
$g\left(  x,t\right)  \in\mathfrak{M}_{2\lambda}$ uniformly in $t$,
$0\leq\lambda<\alpha$, $g(x,t)$ is continuous in $\left(  x,t\right)  $, and
$u:\mathbb{Q}_{p}^{4}\times\left[  0,T\right]  \rightarrow\mathbb{C}$ is an
unknown function. We say that $u\left(  x,t\right)  $ is a \textit{solution}
of (\ref{Cauchy_Pr_1}), if $u\left(  x,t\right)  $ is continuous in $\left(
x,t\right)  $, $u\left(  \cdot,t\right)  \in Dom\left(  \mathbf{f}\right)  $
for $t\in\left[  0,T\right]  $, $u\left(  x,\cdot\right)  $ is continuously
differentiable for $t\in\left(  0,T\right]  $, $u\left(  x,t\right)
\in\mathfrak{M}_{2\lambda}$ uniformly in $t$, and $u$ satisfies
(\ref{Cauchy_Pr_1}) for all $t>0$.

\begin{theorem}
\label{Thm1} The function
\[
u(x,t)=%
{\displaystyle\int\limits_{\mathbb{Q}_{p}^{4}}}
Z(x-y,t)\varphi(y)d^{4}y+%
{\displaystyle\int\limits_{0}^{t}}
\left(
{\displaystyle\int\limits_{\mathbb{Q}_{p}^{4}}}
Z(x-y,t-\theta)g(y,\theta)d^{4}y\right)  d\theta
\]
is a solution of Cauchy problem (\ref{Cauchy_Pr_1}).
\end{theorem}

The proof of the theorem will be accomplished through the following lemmas.

\begin{lemma}
\label{lemma2}Assume that $g\in\mathfrak{M}_{2\lambda}$, $0\leq\lambda<\alpha
$, uniformly with respect to $\theta$. Then the function
\[
u_{2}(x,t,\tau):=%
{\displaystyle\int\limits_{\tau}^{t}}
\left(
{\displaystyle\int\limits_{\mathbb{Q}_{p}^{4}}}
Z(x-y,t-\theta)g(y,\theta)d^{4}y\right)  d\theta
\]
belongs to $\mathfrak{M}_{2\lambda}$ uniformly with respect to $t$ and$\ \tau$.
\end{lemma}

\begin{proof}
We first note that $u_{2}(x,t,\tau)$ has the same exponent of local constancy
as $g$, and thus it does not depend on $t$ and$\ \tau$. We now show that
$|u_{2}(x,t,\tau)|\leq C_{0}(1+||x||_{p}^{2\lambda})$. By Proposition
\ref{Prop2} (i),%
\begin{align*}
|u_{2}(x,t,\tau)|  &  \leq%
{\displaystyle\int\limits_{\tau}^{t}}
\left(
{\displaystyle\int\limits_{\mathbb{Q}_{p}^{4}}}
|Z(x-y,t-\theta)||g(y,\theta)|d^{4}y\right)  d\theta\\
&  \leq C_{1}%
{\displaystyle\int\limits_{\tau}^{t}}
(t-\theta)\left(
{\displaystyle\int\limits_{\mathbb{Q}_{p}^{4}}}
((t-\theta)^{1/2\alpha}+||x-y||_{p})^{-2\alpha-4}(1+||y||_{p}^{2\lambda
})dy\right)  d\theta.
\end{align*}
Now the result follows from the following estimation.

\textbf{Assertion}(\cite[Proposition 2]{R-Zu}).\textbf{ }If $b>0$,
$0\leq\lambda<2\alpha$, and $x\in\mathbb{Q}_{p}^{4},$ then
\begin{equation}%
{\displaystyle\int\limits_{\mathbb{Q}_{p}^{4}}}
\left(  b+||x-\xi||_{p}\right)  ^{-2\alpha-4}||\xi||_{p}^{2\lambda}d^{4}%
\xi\leq Cb^{-2\alpha}\left(  1+||x||_{p}^{2\lambda}\right)  , \label{Key_enq}%
\end{equation}
where the constant $C$ does not depend on $b$ or $x$.
\end{proof}

\begin{lemma}
\label{lemma4} Assume that $g\in\mathfrak{M}_{2\lambda}$, $0\leq\lambda
<\alpha$, uniformly with respect to $\theta$. Then
\[
\frac{\partial u_{2}\left(  x,t,\tau\right)  }{\partial t}=g(x,t)+%
{\displaystyle\int\limits_{\tau}^{t}}
\left(
{\displaystyle\int\limits_{\mathbb{Q}_{p}^{4}}}
\frac{\partial Z(x-\xi,t-\theta)}{\partial t}[g(\xi,\theta)-g(x,\theta
)]d^{4}\xi\right)  d\theta.
\]

\end{lemma}

\begin{proof}
Set
\[
u_{2,h}(x,t,\tau):=%
{\displaystyle\int\limits_{\tau}^{t-h}}
d\theta%
{\displaystyle\int\limits_{\mathbb{Q}_{p}^{4}}}
Z(x-\xi,t-\theta)g(\xi,\theta)d^{4}\xi,
\]
where $h$ is a small positive number. By differentiating $u_{h}$ under the
sign of integral
\begin{align*}
\frac{\partial u_{2,h}}{\partial t}  &  =%
{\displaystyle\int\limits_{\tau}^{t-h}}
d\theta%
{\displaystyle\int\limits_{\mathbb{Q}_{p}^{4}}}
\frac{\partial Z(x-\xi,t-\theta)}{\partial t}g(\xi,\theta)d^{4}\xi+%
{\displaystyle\int\limits_{\mathbb{Q}_{p}^{4}}}
Z(x-\xi,h)g(\xi,t-h)d^{4}\xi\\
&  =%
{\displaystyle\int\limits_{\tau}^{t-h}}
d\theta%
{\displaystyle\int\limits_{\mathbb{Q}_{p}^{4}}}
\frac{\partial Z(x-\xi,t-\theta)}{\partial t}[g(\xi,\theta)-g(x,\theta
)]d^{4}\xi\\
&  +%
{\displaystyle\int\limits_{\tau}^{t-h}}
g(x,\theta)d\theta%
{\displaystyle\int\limits_{\mathbb{Q}_{p}^{4}}}
\frac{\partial Z(x-\xi,t-\theta)}{\partial t}d^{4}\xi+%
{\displaystyle\int\limits_{\mathbb{Q}_{p}^{4}}}
Z(x-\xi,h)[g(\xi,t-h)-g(\xi,t)]d^{4}\xi\\
&  +%
{\displaystyle\int\limits_{\mathbb{Q}_{p}^{4}}}
Z(x-\xi,h)g(\xi,t)d^{4}\xi.
\end{align*}
The first integral contains no singularity at $t=\theta$ due to Proposition
\ref{Prop2} (iii) and the local constancy of $g$. By Proposition \ref{Prop2}
(i) and Corollary \ref{cor2} (i), the second integral is equal to zero. The
third integral can be written as the sum of the integrals over $\{\xi
\in\mathbb{Q}_{p}^{4}\mid||x-\xi||_{p}\leq p^{M}\}$, where $M$ is the exponent
of local constancy of $g$, and the complement of this set. The first integral
tends to zero when $h\rightarrow0^{+}$ due to the uniform local constancy of
$g$, while the other tends to zero when $h\rightarrow0^{+}$ due to Proposition
\ref{Prop2} (i) and condition $\lambda<\alpha$. Finally, the fourth integral
tends to $g(x,t)$ as $h\rightarrow0^{+}$, c.f. Lemma \ref{lemma3}.
\end{proof}

For $\varphi\in\mathfrak{M}_{2\lambda}$, $0\leq\lambda<\alpha$, we set%
\[
u_{1}(x,t):=%
{\displaystyle\int\limits_{\mathbb{Q}_{p}^{4}}}
Z(x-y,t)\varphi(y)d^{4}y\text{ for }t>0.
\]

\begin{lemma}
\label{lemma5}Assume that $\varphi\in\mathfrak{M}_{2\lambda}$, $0\leq
\lambda<\alpha$, then the following assertions hold:

\noindent(i) $u_{1}(x,t)$ belongs to $\mathfrak{M}_{2\lambda}$ uniformly with
respect to $t$;

\noindent(ii) $\frac{\partial u_{1}}{\partial t}(x,t)=\int_{\mathbb{Q}_{p}%
^{4}}\frac{\partial Z}{\partial t}(x-y,t)\varphi(y)d^{4}y$ for $t>0$.
\end{lemma}

\begin{proof}
(i) The proof is similar to that of Lemma \ref{lemma2}.

(ii) By Proposition \ref{Prop2} (ii),
\begin{align*}
\lim_{h\rightarrow0}\frac{u_{1}(x,t+h)-u_{1}(x,t)}{h}  &  =\lim_{h\rightarrow
0}%
{\displaystyle\int\limits_{\mathbb{Q}_{p}^{4}}}
\left[  \frac{Z(x-y,t+h)-Z\left(  x-y,t\right)  }{h}\right]  \varphi
(y)d^{4}y\\
&  =\lim_{h\rightarrow0}%
{\displaystyle\int\limits_{\mathbb{Q}_{p}^{4}}}
\frac{\partial Z}{\partial t}(x-y,\tau)\varphi(y)d^{4}y
\end{align*}
where $\tau$ is between $t$ and $t+h$. Now the result follows from Proposition
\ref{Prop2} (iii) by applying the Dominated Convergence Theorem.
\end{proof}

\begin{lemma}
\label{lemma6}Assume that $\lambda<\gamma\leq\alpha$. Then
\[
(\mathbf{f}(\partial,\gamma)u_{1})(x,t)=%
{\displaystyle\int\limits_{\mathbb{Q}_{p}^{4}}}
(\mathbf{f}(\partial,\gamma)Z_{t})(x-y)\varphi(y)d^{4}y\text{ for }t>0.
\]

\end{lemma}

\begin{proof}
By Lemma \ref{lemma5} (i), $u_{1}(x,t)$ belongs to the domain of
$\mathbf{f}(\partial,\gamma)$ for $t>0$ and for $\lambda<\gamma\leq\alpha$,
then for any $L\in\mathbb{N}$, the following integral exists:
\begin{align*}
&  \frac{1-p^{\gamma}}{1-p^{-\gamma-2}}%
{\displaystyle\int\limits_{||y||_{p}>p^{-L}}}
|f(y)|_{p}^{-\gamma-2}\left[  u_{1}(x-y,t)-u_{1}(x,t)\right]  d^{4}y\\
&  =\frac{1-p^{\gamma}}{1-p^{-\gamma-2}}%
{\displaystyle\int\limits_{||y||_{p}>p^{-L}}}
|f(y)|_{p}^{-\gamma-2}\left[  \int_{\mathbb{Q}_{p}^{4}}\left[  Z_{t}%
(x-y-\xi)-Z_{t}(x-\xi)\right]  \varphi(\xi)d^{4}\xi\right]  d^{4}y.
\end{align*}
By using Fubini's Theorem, see (\ref{Zuniga_ineq}), Proposition \ref{Prop2}
(i),
\begin{align}
&
{\displaystyle\int\limits_{\mathbb{Q}_{p}^{4}}}
\left[  \frac{1-p^{\gamma}}{1-p^{-\gamma-2}}%
{\displaystyle\int\limits_{||y||_{p}>p^{-L}}}
|f(y)|_{p}^{-\gamma-2}\left[  Z_{t}(x-y-\xi)-Z_{t}(x-\xi)\right]
d^{4}y\right]  \varphi(\xi)d^{4}\xi\nonumber\\
&  =:%
{\displaystyle\int\limits_{\mathbb{Q}_{p}^{4}}}
\varphi(\xi)Z_{t}^{\left(  \gamma,L\right)  }(x-\xi)d^{4}\xi.
\label{integral0}%
\end{align}
By fixing a positive integer $M$, the last integral in (\ref{integral0})\ can
be expressed as
\begin{equation}%
{\displaystyle\int\limits_{||x-\xi||_{p}\geq p^{-M}}}
\varphi(\xi)Z_{t}^{\left(  \gamma,L\right)  }(x-\xi)d^{4}\xi+%
{\displaystyle\int\limits_{||x-\xi||_{p}<p^{-M}}}
\varphi(\xi)Z_{t}^{\left(  \gamma,L\right)  }(x-\xi)d^{4}\xi.
\label{integral1}%
\end{equation}
note that if $||x||_{p}\geq p^{-M}$ and $M<L-1$, then, by Proposition
\ref{Prop1} (iii), $Z_{t}^{\left(  \gamma,L\right)  }=(\mathbf{f}%
(\partial,\gamma)Z_{t})(x)$, and
\begin{align*}
&  \lim_{L\rightarrow\infty}\frac{1-p^{\gamma}}{1-p^{-\gamma-2}}%
{\displaystyle\int\limits_{||y||_{p}>p^{-L}}}
|f(y)|_{p}^{-\gamma-2}\left[  u_{1}(x-y,t)-u_{1}(x,t)\right]  d^{4}y\\
&  =%
{\displaystyle\int\limits_{||x-\xi||_{p}\geq p^{-M}}}
\varphi(\xi)(\mathbf{f}(\partial,\gamma)Z_{t})(x-\xi)d^{4}\xi+\lim
_{L\rightarrow\infty}%
{\displaystyle\int\limits_{||x-\xi||_{p}<p^{-M}}}
\varphi(\xi)Z_{t}^{\left(  \gamma,L\right)  }(x-\xi)d^{4}\xi,
\end{align*}
for $M<L-1$. Now by using twice Fubini's theorem, see (\ref{Zuniga_ineq}),
Proposition \ref{Prop2} (i), and Proposition \ref{Prop1} (iii), we have
\begin{align*}
&  \lim_{L\rightarrow\infty}%
{\displaystyle\int\limits_{||x-\xi||_{p}<p^{-M}}}
\varphi(\xi)Z_{t}^{\left(  \gamma,L\right)  }(x-\xi)d^{4}\xi=\lim
_{L\rightarrow\infty}%
{\displaystyle\int\limits_{||x-\xi||_{p}<p^{-M}}}
\varphi(\xi)\times\\
&  \left[  \frac{1-p^{\gamma}}{1-p^{-\gamma-2}}%
{\displaystyle\int\limits_{||y||_{p}>p^{-L}}}
|f(y)|_{p}^{-\gamma-2}\left[  Z_{t}(x-y-\xi)-Z_{t}(x-\xi)\right]
d^{4}y\right]  d^{4}\xi\\
&  =\lim_{L\rightarrow\infty}%
{\displaystyle\int\limits_{||y||_{p}>p^{-L}}}
|f(y)|_{p}^{-\gamma-2}\times\\
&  \left[  \frac{1-p^{\gamma}}{1-p^{-\gamma-2}}%
{\displaystyle\int\limits_{||x-\xi||_{p}<p^{-M}}}
\varphi(\xi)\left[  Z_{t}(x-y-\xi)-Z_{t}(x-\xi)\right]  d^{4}\xi\right]
d^{4}y\\
&  =%
{\displaystyle\int\limits_{\mathbb{Q}_{p}^{4}}}
|f(y)|_{p}^{-\gamma-2}\left[  \frac{1-p^{\gamma}}{1-p^{-\gamma-2}}%
{\displaystyle\int\limits_{||x-\xi||_{p}<p^{-M}}}
\varphi(\xi)\left[  Z_{t}(x-y-\xi)-Z_{t}(x-\xi)\right]  d^{4}\xi\right]
d^{4}y\\
&  =\int_{||x-\xi||_{p}<p^{-M}}\varphi(\xi)\left[  \frac{1-p^{\gamma}%
}{1-p^{-\gamma-2}}%
{\displaystyle\int\limits_{\mathbb{Q}_{p}^{4}}}
|f(y)|_{p}^{-\gamma-2}\left[  Z_{t}(x-y-\xi)-Z_{t}(x-\xi)\right]
d^{4}y\right]  d^{4}\xi\\
&  =\int_{||x-\xi||_{p}<p^{-M}}(\mathbf{f}(\partial,\gamma)Z_{t}%
)(x-\xi)\varphi(\xi)d^{4}\xi.
\end{align*}

\end{proof}

\begin{lemma}
\label{lemma7}If $\lambda<\gamma\leq\alpha$, then
\[
(\mathbf{f}(\partial,\gamma)u_{2})(x,t,\tau)=%
{\displaystyle\int\limits_{\tau}^{t}}
\left(
{\displaystyle\int\limits_{\mathbb{Q}_{p}^{4}}}
(\mathbf{f}(\partial,\gamma)Z)(x-y,t-\theta)g(y,\theta)d^{4}y\right)
d\theta\text{ for }t>0\text{.}%
\]

\end{lemma}

\begin{proof}
Let
\[
u_{2,h}(x,t,\tau)=:%
{\displaystyle\int\limits_{\tau}^{t-h}}
\left(
{\displaystyle\int\limits_{\mathbb{Q}_{p}^{4}}}
Z(x-y,t-\theta)g(y,\theta)d^{4}y\right)  d\theta
\]
where $h$ is a small positive number such that $0<h<t-\tau$. Set
\[
Z^{\left(  \gamma,L\right)  }(x,t)=\frac{1-p^{\gamma}}{1-p^{-\gamma-2}}%
{\displaystyle\int\limits_{||y||_{p}>p^{-L}}}
|f(y)|_{p}^{-\gamma-2}\left[  Z(x-y,t)-Z(x,t)\right]  d^{4}y.
\]
By the Fubini Theorem
\begin{align}
&  \frac{1-p^{\gamma}}{1-p^{-\gamma-2}}%
{\displaystyle\int\limits_{||y||_{p}>p^{-L}}}
|f(y)|_{p}^{-\gamma-2}\left[  u_{2,h}(x-y,t,\tau)-u_{2,h}(x,t,\tau)\right]
d^{4}y\nonumber\\
&  =%
{\displaystyle\int\limits_{\tau}^{t-h}}
{\displaystyle\int\limits_{\mathbb{Q}_{p}^{4}}}
Z^{\left(  \gamma,L\right)  }(x-\xi,t-\theta)g(\xi,\theta)d^{4}\xi d\theta.
\label{Eq_4_20}%
\end{align}
Note that
\begin{multline*}
Z^{\left(  \gamma,L\right)  }(x,t)=\frac{1-p^{\gamma}}{1-p^{-\gamma-2}}%
{\displaystyle\int\limits_{||y||_{p}>p^{-L}}}
\times\\
\left[
{\displaystyle\int\limits_{\mathbb{Q}_{p}^{4}}}
|f(y)|_{p}^{-\gamma-2}\chi(\xi\cdot x)\left[  \chi(-\xi\cdot y)-1\right]
e^{-\kappa t|f^{\circ}(\xi)|_{p}^{\alpha}}d^{4}\xi\right]  d^{4}y\\
=%
{\displaystyle\int\limits_{\mathbb{Q}_{p}^{4}}}
\chi(\xi\cdot x)e^{-\kappa t|f^{\circ}(\xi)|_{p}^{\alpha}}\left[
\frac{1-p^{\gamma}}{1-p^{-\gamma-2}}%
{\displaystyle\int\limits_{||y||_{p}>p^{-L}}}
|f(y)|_{p}^{-\gamma-2}\left[  \chi(-\xi\cdot y)-1\right]  d^{4}y\right]
d^{4}\xi\\
=%
{\displaystyle\int\limits_{\mathbb{Q}_{p}^{4}}}
\chi(\xi\cdot x)e^{-at|f^{\circ}(\xi)|_{p}^{\alpha}}P_{L}(\xi)d^{4}\xi,
\end{multline*}
where
\[
{P_{L}(\xi)=\frac{1-p^{\gamma}}{1-p^{-\gamma-2}}\int\limits_{||y||_{p}>p^{-L}%
}|f(y)|_{p}^{-\gamma-2}\left[  \chi(-\xi\cdot y)-1\right]  d^{4}y}.
\]
On the other hand, by (\ref{Zuniga_ineq}),
\[
\left\vert P_{L}(\xi)\right\vert \leq B^{-\gamma-2}\left\vert \frac
{1-p^{\gamma}}{1-p^{-\gamma-2}}\right\vert \int\limits_{||y||_{p}>p^{-L}%
}||y||_{p}^{-2\gamma-4}\left\vert \chi(-\xi\cdot y)-1\right\vert d^{4}y,
\]
and by using a similar reasoning to the one used in \cite[p. 142]{Koch}, we
have
\[
\left\vert P_{L}(\xi)\right\vert \leq C||\xi||_{p}^{2\gamma}%
\]
whence
\begin{equation}
\left\vert Z^{\left(  \gamma,L\right)  }(x,t)\right\vert \leq%
{\displaystyle\int\limits_{\mathbb{Q}_{p}^{4}}}
e^{-\kappa t|f^{\circ}(\xi)|_{p}^{\alpha}}\left\vert P_{L}(\xi)\right\vert
d^{4}\xi\leq C%
{\displaystyle\int\limits_{\mathbb{Q}_{p}^{4}}}
e^{-\kappa t|f^{\circ}(\xi)|_{p}^{\alpha}}||\xi||_{p}^{2\gamma}d^{4}\leq
C^{\prime}, \label{cota}%
\end{equation}
where $C^{\prime}$ is a positive constant, which not depend on $x,t\geq
h+\tau,L$.

By writing the right-hand side of (\ref{Eq_4_20}) as
\begin{align}
&
{\displaystyle\int\limits_{\tau}^{t-h}}
{\displaystyle\int\limits_{||x-\xi||_{p}\geq p^{-K}}}
Z^{\left(  \gamma,L\right)  }(x-\xi,t-\theta)g(\xi,\theta)d^{4}\xi
d\theta\label{integral}\\
&  +%
{\displaystyle\int\limits_{\tau}^{t-h}}
{\displaystyle\int\limits_{||x-\xi||_{p}<p^{-K}}}
Z^{\left(  \gamma,L\right)  }(x-\xi,t-\theta)g(\xi,\theta)d^{4}\xi
d\theta,\nonumber
\end{align}
where $K$ is a fixed natural number. Now the result follows by taking limit
$L\rightarrow\infty$ in (\ref{Eq_4_20}). Indeed, for the first integral in
(\ref{integral}), if $||x-\xi||_{p}\geq p^{-K}$ and $L>K+1$, then $Z^{\left(
\gamma,L\right)  }(x-\xi,t-\theta)=(\mathbf{f}(\partial,\gamma)Z)(x-\xi
,t-\theta)$. For the second integral in (\ref{integral}), we use (\ref{cota})
and the Dominated Convergence Theorem.
\end{proof}

\subsection{Proof of Theorem \ref{Thm1}}

By Lemmas \ref{lemma2} and \ref{lemma5} (i), $u(x,t)\in\mathfrak{M}_{2\lambda
}$ uniformly with respect to $t$, and by Lemma \ref{lemma3} $u(x,t)$ satisfies
the initial condition. By Lemmas \ref{lemma4}-\ref{lemma7} and Corollary
\ref{cor3}, $u(x,t)$ is a solution of Cauchy problem (\ref{Cauchy_Pr_1}).

\section{\label{Sect.Eq.Var.Coeff.}Parabolic-type equations with variable
coefficients}

Fix $n+1$ positive real numbers satisfying%
\[
0<\alpha_{1}<\alpha_{2}<\cdots<\alpha_{n}<\alpha,
\]
with $\alpha>1$. We also fix $n+2$ functions $a_{k}(x,t)$, $k=0,\ldots,n$ and
$b(x,t)$ from $\mathbb{Q}_{p}^{4}\times\left[  0,T\right]  $ to $\mathbb{R}$,
here $T$ is a fixed positive constant. We assume that $a_{k}(x,t)$,
$k=0,\ldots,n$ and $b(x,t)$ satisfy the following conditions:

\noindent(i) they belong to $\mathfrak{M}_{0}$, with respect to $x$, uniformly
in $t\in\left[  0,T\right]  $;

\noindent(ii) they satisfy the Hölder condition in $t$ with exponent $\nu
\in\left(  0,1\right]  $ uniformly in $x$.

In addition we assume the\ \textit{uniform parabolicity condition}
$a_{0}(x,t)\geq\mu>0$.

We set
\[
\boldsymbol{F}\left(  \partial,\alpha_{1},\alpha_{2},\cdots,\alpha_{n}\right)
:=%
{\displaystyle\sum\limits_{k=1}^{n}}
a_{k}(x,t)\mathbf{f}\left(  \partial,\alpha_{k}\right)  +b(x,t)I.
\]
Note that $\boldsymbol{F}\left(  \partial,\alpha_{1},\alpha_{2},\cdots
,\alpha_{n}\right)  \mathfrak{M}_{2\lambda}\subset\mathfrak{M}_{2\lambda}$ for
$\lambda<\alpha_{1}$, c.f. \ (\ref{Ext_oper}) .

In this section we study the following initial value problem:
\begin{equation}%
\begin{cases}
\dfrac{\partial u(x,t)}{\partial t}+a_{0}(x,t)(\mathbf{f}(\partial
,\alpha)u)(x,t)+\left(  \boldsymbol{F}\left(  \partial,\alpha_{1},\alpha
_{2},\cdots,\alpha_{n}\right)  u\right)  (x,t)=g(x,t)\\
\\
u(x,0)=\varphi(x)\text{, }x\in\mathbb{Q}_{p}^{4},\ t\in(0,T],
\end{cases}
\label{cauchy_Pr_2}%
\end{equation}
where \ $x\in\mathbb{Q}_{p}^{4},\ t\in(0,T]$, $\varphi\in\mathfrak{M}%
_{2\lambda}$ with $0\leq\lambda<\alpha_{1}$ (if $a_{1}(x,t)=\dots
=a_{n}(x,t)\equiv0$, then we shall assume that $0\leq\lambda<\alpha$), and
$g(x,t)$ is continuous in $(x,t)$, and $g\left(  x,t\right)  \in
\mathfrak{M}_{2\lambda}$ uniformly in $t\in\left[  0,T\right]  $,
$0\leq\lambda<\alpha_{1}$.

In this section we find the solution of the general problem (\ref{cauchy_Pr_2}%
). The technique used here is an adaptation of the classical Archimedean
techniques, see e.g. \cite{Fri}, \cite{I-K-O}. In the $p$-adic setting the
technique was introduced by A. N. Kochubei in \cite{Koch0}. Our presentation
is highly influenced by Kochubei's book \cite{Koch}. The proofs of some
theorems are very similar to the corresponding in \cite{Koch} for this reason
we will omit them.

The first step of the construction of a fundamental solution is to study the
parametrized fundamental solution $Z(x,t,y,\theta)$ for the Cauchy problem:
\[
\left\{
\begin{array}
[c]{l}%
\frac{\partial u(x,t)}{\partial t}+a_{0}(y,\theta)(\mathbf{f}(\partial
,\alpha)u)(x,t)=0\\
\\
u\left(  x,0\right)  =\varphi\left(  x\right)  ,
\end{array}
\right.
\]
where $y\in\mathbb{Q}_{p}^{4}$ and $\theta>0$ are parameters. By the results
of Section \ref{Sect.Heat.Kernel}, we have
\[
Z(x,t,y,\theta)=%
{\displaystyle\int\limits_{\mathbb{Q}_{p}^{4}}}
\chi(\xi\cdot x)e^{-a_{0}(y,\theta)t|f^{\circ}(\xi)|_{p}^{\alpha}}d^{4}\xi,
\]
and if $x\neq0$, then
\[
{Z(x,t,y,\theta)=\sum_{m=1}^{\infty}\frac{(-1)^{m}}{m!}\left(  \frac
{1-p^{\alpha m}}{1-p^{-\alpha m-2}}\right)  (a_{0}(y,\theta)t)^{m}%
|f(x)|_{p}^{-\alpha m-2},}%
\]
c.f. Proposition \ref{Prop1} (ii). By the Propositions \ref{Prop2},
\ref{Prop3}, \ref{Prop4}, and Corollaries \ref{cor2} (i) and \ref{cor4}, we
have%
\begin{equation}
Z(x,t,y,\theta)\leq Ct(t^{1/2\alpha}+||x||_{p})^{-2\alpha-4},\qquad
x\in\mathbb{Q}_{p}^{4},\quad t>0, \label{fund_sol_par}%
\end{equation}%
\begin{equation}
|(\mathbf{f}(\partial,\gamma)Z)(x,t,y,\theta)|\leq C(t^{1/2\alpha}%
+||x||_{p})^{-2\gamma-4},\qquad x\in\mathbb{Q}_{p}^{4},\quad t>0,
\label{f_oper_par}%
\end{equation}%
\begin{equation}
\dfrac{\partial Z(x,t,y,\theta)}{\partial t}=-a_{0}(y,\theta){%
{\displaystyle\int\limits_{\mathbb{Q}_{p}^{4}}}
}|f^{\circ}(\eta)|_{p}^{\alpha}e^{-a_{0}(y,\theta)t|f^{\circ}(\eta)|^{\alpha}%
}\chi(x\cdot\eta)d^{4}\eta, \label{parZ_par}%
\end{equation}%
\begin{equation}
\left\vert \dfrac{\partial Z(x,t,y,\theta)}{\partial t}\right\vert \leq
C\left(  t^{1/2\alpha}+||x||_{p}\right)  ^{-2\alpha-4}, \label{est_par_Z-par}%
\end{equation}%
\begin{equation}
(\mathbf{f}(\partial,\gamma)Z)(x,t,y,\theta)=\int_{\mathbb{Q}_{p}^{4}%
}|f^{\circ}(\eta)|_{p}^{\gamma}e^{-a_{0}(y,\theta)t|f^{\circ}(\eta)|^{\alpha}%
}\chi(x\cdot\eta)d^{4}\eta,\quad x\neq0,\quad0<\gamma\leq\alpha,
\label{for_f_par}%
\end{equation}%
\begin{equation}
{%
{\displaystyle\int\limits_{\mathbb{Q}_{p}^{4}}}
}Z(x,t,y,\theta)d^{4}x=1, \label{vol_Z_par}%
\end{equation}%
\begin{equation}
{%
{\displaystyle\int\limits_{\mathbb{Q}_{p}^{4}}}
}(\mathbf{f}(\partial,\gamma)Z)(x,t,y,\theta)d^{4}x=0, \label{int_f_par}%
\end{equation}
where the constants do not depend on $y,\theta$.

\begin{lemma}
\label{lemma8A}There exists a positive constant $C$, such that%
\begin{equation}
\left\vert {%
{\displaystyle\int\limits_{\mathbb{Q}_{p}^{4}}}
}\dfrac{\partial Z(x-y,t,y,\theta)}{\partial t}d^{4}y\right\vert \leq C.
\label{estimacio}%
\end{equation}

\end{lemma}

\begin{proof}
The proof follows from (\ref{vol_Z_par}), (\ref{parZ_par}),
(\ref{est_par_Z-par}) by using the resoning given in \cite{Koch} for Lemma 4.5.
\end{proof}

Consider the parametrized heat potential
\[
u(x,t,\tau):=\int\limits_{\tau}^{t}\int\limits_{\mathbb{Q}_{p}^{4}%
}Z(x-y,t-\theta,y,\theta)g(y,\theta)d^{4}yd\theta,
\]
with $g\in\mathfrak{M}_{2\lambda},0\leq\lambda<\alpha_{1},$ uniformly with
respect to $\theta$, and continuous in $(y,\theta)$. By using
(\ref{fund_sol_par}) we obtain as Lemma \ref{lemma2} that $u(x,t,\tau)$ is
locally constant and belongs to $\mathfrak{M}_{2\lambda}$ uniformly with
respect to $t,\ \tau$.\newline Also we have
\begin{align}
\frac{\partial u}{\partial t}(x,t,\tau)  &  =g(x,t)+\int\limits_{\tau}^{t}%
{\displaystyle\int\limits_{\mathbb{Q}_{p}^{4}}}
\frac{\partial Z(x-y,t-\theta,y,\theta)}{\partial t}\left[  g(y,\theta
)-g(x,\theta)\right]  d^{4}yd\theta\text{ }\nonumber\\
&  +\int\limits_{\tau}^{t}%
{\displaystyle\int\limits_{\mathbb{Q}_{p}^{4}}}
\frac{\partial Z(x-y,t-\theta,y,\theta)}{\partial t}g(x,\theta)d^{4}yd\theta,
\label{derivadas}%
\end{align}
see Lemma \ref{lemma8A}, and
\begin{equation}
(\mathbf{f}(\partial,\gamma)u)(x,t,\tau)=\int\limits_{\tau}^{t}%
{\displaystyle\int\limits_{\mathbb{Q}_{p}^{4}}}
Z^{\left(  \gamma\right)  }(x-y,t-\theta,y,\theta)g(y,\theta)d^{4}yd\theta,
\label{derivadas1}%
\end{equation}
with $Z^{\left(  \gamma\right)  }\left(  x,t,y,\theta\right)  :=\mathbf{f}%
(\partial,\gamma)Z\left(  x,t,y,\theta\right)  $, for $\lambda<\gamma<\alpha$,
and
\begin{align}
(\mathbf{f}(\partial,\alpha)u)(x,t,\tau)  &  =\int\limits_{\tau}^{t}%
{\displaystyle\int\limits_{\mathbb{Q}_{p}^{4}}}
Z^{\left(  \alpha\right)  }(x-y,t-\theta,y,\theta)\left[  g(y,\theta
)-g(x,\theta)\right]  d^{4}yd\theta\nonumber\\
&  +\int\limits_{\tau}^{t}%
{\displaystyle\int\limits_{\mathbb{Q}_{p}^{4}}}
\left[  Z^{\left(  \alpha\right)  }(x-y,t-\theta,y,\theta)-Z^{\left(
\alpha\right)  }(x-y,t-\theta,x,\theta)\right]  g(x,\theta)d^{4}yd\theta.
\label{derivadas2}%
\end{align}
As in \cite{Koch} we look for a fundamental solution of (\ref{cauchy_Pr_2}) of
the form
\begin{equation}
\Gamma(x,t,\xi,\tau)=Z(x-\xi,t-\tau,\xi,\tau)+\int\limits_{0}^{t}%
{\displaystyle\int\limits_{\mathbb{Q}_{p}^{4}}}
Z(x-\eta,t-\theta,\eta,\theta)\phi(\eta,\theta,\xi,\tau)d^{4}\eta d\theta.
\label{gammaequation}%
\end{equation}
By using formally the formulas given above, we can see that $\phi(x,t,\xi
,\tau)$ is a solution of the integral equation
\begin{equation}
\phi(x,t,\xi,\tau)=R(x,t,\xi,\tau)+\int\limits_{0}^{t}%
{\displaystyle\int\limits_{\mathbb{Q}_{p}^{4}}}
R(x,t,\eta,\theta)\phi(\eta,\theta,\xi,\tau)d^{4}\eta d\theta\label{ecuafi}%
\end{equation}
where
\begin{align*}
R(x,t,\xi,\tau)  &  =\left[  a_{0}(\xi,\tau)-a_{0}(x,t)\right]  Z^{\left(
\alpha\right)  }(x-\xi,t-\tau,\xi,\tau)\\
&  -\sum_{k=1}^{n}a_{k}(x,t)Z^{\left(  \alpha_{k}\right)  }(x-\xi,t-\tau
,\xi,\tau)-b(x,t)Z(x-\xi,t-\tau,\xi,\tau).
\end{align*}
Integral equation (\ref{ecuafi}) can be solved by the methods of successive
approximations:%
\begin{equation}
\phi(x,t,\xi,\tau)=\sum_{m=1}^{\infty}R_{m}(x,t,\xi,\tau)
\label{sol_suc_aprox}%
\end{equation}
where
\[
R_{1}(x,t,\xi,\tau):=R(x,t,\xi,\tau)
\]
and
\[
R_{m+1}(x,t,\xi,\tau):=\int\limits_{0}^{t}%
{\displaystyle\int\limits_{\mathbb{Q}_{p}^{4}}}
R(x,t,\eta,\theta)R_{m}(\eta,\theta,\xi,\tau)d^{4}\eta d\theta.
\]
In order to prove the convergence of series (\ref{sol_suc_aprox}), we need the
following lemma which is an easy variation of the Lemmas 6 and 7 given in
\cite{R-V}.

We set $\alpha_{n+1}:=\alpha\left(  1-\nu\right)  >\alpha_{n}$.

\begin{lemma}
\label{lemma8}The following estimations hold:%
\[
\left\vert R(x,t,\xi,\tau)\right\vert \leq C\sum_{k=1}^{n+1}\left(  \left(
t-\tau\right)  ^{1/2\alpha}+||x-\xi||_{p}\right)  ^{-2\alpha_{k}-4}%
\]
and%
\[
\left\vert R_{m}(x,t,\xi,\tau)\right\vert \leq DM^{m}\dfrac{\Gamma\left(
\frac{\nu}{2\alpha}\right)  ^{m}}{\Gamma\left(  \frac{m\nu}{2\alpha}\right)
}\sum_{k=1}^{n+1}\left(  \left(  t-\tau\right)  ^{1/2\alpha}+||x-\xi
||_{p}\right)  ^{-2\alpha_{k}-4}%
\]
where $C$, $M$ and $D$ are a positive constants and $\Gamma\left(
\cdot\right)  $ is the Archimedean Gamma function.
\end{lemma}

Lemma \ref{lemma8}\ also implies
\begin{equation}
\left\vert \phi(x,t,\xi,\tau)\right\vert \leq C\sum_{k=1}^{n+1}\left(  \left(
t-\tau\right)  ^{1/2\alpha}+||x-\xi||_{p}\right)  ^{-2\alpha_{k}-4}.
\label{desigualdadphi}%
\end{equation}

\begin{theorem}
\label{Thm2}The function
\begin{equation}
u(x,t)=%
{\displaystyle\int\limits_{\mathbb{Q}_{p}^{4}}}
\Gamma(x,t,\xi,0)\varphi(\xi)d^{4}\xi+\int\limits_{0}^{t}%
{\displaystyle\int\limits_{\mathbb{Q}_{p}^{4}}}
\Gamma(x,t,\xi,\tau)g(\xi,\tau)d^{4}\xi d\tau\label{solucion}%
\end{equation}
which is continuous on $\mathbb{Q}_{p}^{n}\times\lbrack0;T]$, continuously
differentiable in $t\in\left(  0,T\right]  $, and belonging to $\mathfrak{M}%
_{2\lambda}$ uniformly with respect to $t$ is a solution of Cauchy problem
(\ref{cauchy_Pr_2}). The fundamental solution $\Gamma(x,t,\xi,\tau)$,
$x,\xi\in\mathbb{Q}_{p}^{4}$, $0\leq\tau<t\leq T$, is of the form
\begin{equation}
\Gamma(x,t,\xi,\tau)=Z(x-\xi,t-\tau,\xi,\tau)+W(x,t,\xi,\tau) \label{gamma1}%
\end{equation}
with
\begin{align}
\left\vert W(x,t,\xi,\tau)\right\vert  &  \leq C\left(  t-\tau\right)
^{1+\nu}\left[  \left(  t-\tau\right)  ^{1/2\alpha}+||x-\xi||_{p}\right]
^{-2\alpha-4}\label{w1}\\
&  +C\left(  t-\tau\right)  \sum_{k=1}^{n+1}\left[  \left(  t-\tau\right)
^{1/2\alpha}+||x-\xi||_{p}\right]  ^{-2\alpha_{k}-4}.\nonumber
\end{align}
Furthermore $Z(x,t,y,\theta)$ satisfies the estimates (\ref{fund_sol_par}),
(\ref{f_oper_par}), (\ref{est_par_Z-par}), (\ref{estimacio}).
\end{theorem}

\begin{proof}
Denote for $u_{1}(x,t)$ and $u_{2}(x,t)$ the first and second summands in the
right-hand side of (\ref{solucion}). Substituting (\ref{gammaequation}) into
(\ref{solucion}) we get and
\begin{align}
u_{1}(x,t)  &  =%
{\displaystyle\int\limits_{\mathbb{Q}_{p}^{4}}}
Z(x-\xi,t,\xi,0)\varphi(\xi)d^{4}\xi\label{u1}\\
&  +\int_{0}^{t}%
{\displaystyle\int\limits_{\mathbb{Q}_{p}^{4}}}
Z(x-\eta,t-\theta,\eta,\theta)G(\eta,\theta)d^{4}\eta d\theta,\nonumber
\end{align}
where
\[
G(\eta,\theta)=%
{\displaystyle\int\limits_{\mathbb{Q}_{p}^{4}}}
\phi(\eta,\theta,\xi,0)\varphi(\xi)d^{4}\xi d\tau.
\]
and%
\begin{align}
u_{2}(x,t)  &  =\int_{0}^{t}%
{\displaystyle\int\limits_{\mathbb{Q}_{p}^{4}}}
Z(x-\xi,t-\tau,\xi,\tau)g(\xi,\tau)d^{4}\xi d\tau\label{u2}\\
&  +\int_{0}^{t}%
{\displaystyle\int\limits_{\mathbb{Q}_{p}^{4}}}
Z(x-\eta,t-\theta,\eta,\theta)F(\eta,\theta)d^{4}\eta d\theta,\nonumber
\end{align}

where%
\[
F(\eta,\theta)=\int_{0}^{\theta}%
{\displaystyle\int\limits_{\mathbb{Q}_{p}^{4}}}
\phi(\eta,\theta,\xi,\tau)g(\xi,\tau)d^{4}\xi d\tau.
\]

Now by (\ref{desigualdadphi}) and (\ref{Key_enq}),
\[
|F(\eta,\theta)|\leq C,\text{ and }|G(\eta,\theta)|\leq C\theta^{-\alpha
_{n+1}/\alpha}%
\]
for all $\eta\in\mathbb{Q}_{p}^{4}$ and $\theta\in(0,T]$. In addition the
functions $F$ and $G$ are uniformly locally constant. Indeed, by the recursive
definition of the function $\phi$, we see that if $N$ is a local constancy
exponent for all the functions $g$, $\varphi$, $a_{k}$, $b$, $Z^{\left(
\alpha_{k}\right)  }$ and $Z$, and if $|\delta|\leq p^{-N}$, then
\[
\phi(\eta+\delta,\theta,\xi+\delta,\tau)=\phi(\eta,\theta,\xi,\tau),
\]
therefore
\[
F(\eta+\delta,\theta)=F(\eta,\theta),\text{ and }G(\eta+\delta,\theta
)=G(\eta,\theta).
\]
Thus, the potentials in the expressions for $u_{1}(x,t)$ and $u_{2}(x,t)$
satisfy the conditions under which the differentiation formulas
(\ref{derivadas}), (\ref{derivadas1}), (\ref{derivadas2}) were obtained. By
using these formulas one verifies after some simple transformations that
$u(x,t)$ is a solution of the equation in (\ref{cauchy_Pr_2}).

We now show that $u(x,t)\rightarrow\varphi(x)$ as $t\rightarrow0^{+}$. Due to
(\ref{u1}) and (\ref{u2}), it is sufficient to verify that
\[
{v(x,t)=%
{\displaystyle\int\limits_{\mathbb{Q}_{p}^{4}}}
Z(x-\xi,t,\xi,0)\varphi(\xi)d^{4}\xi\rightarrow\varphi(x)}\text{ as
}t\rightarrow0^{+}.
\]
By virtue of equation (\ref{vol_Z_par}),
\begin{align*}
v(x,t)  &  ={%
{\displaystyle\int\limits_{\mathbb{Q}_{p}^{4}}}
}\left[  Z(x-\xi,t,\xi,0)-Z(x-\xi,t,x,0)\right]  \varphi(\xi)d^{4}\xi\\
&  +{%
{\displaystyle\int\limits_{\mathbb{Q}_{p}^{4}}}
}Z(x-\xi,t,x,0)\left[  \varphi(\xi)-\varphi(x)\right]  d^{4}\xi+\varphi(x).
\end{align*}
We now use that $Z$ (as a function of the its third argument) and $\varphi$
are locally constant, then the integrals in the previous formula are performed
over the set
\[
\left\{  \xi\in\mathbb{Q}_{p}^{4};||x-\xi||_{p}\geq p^{-N}\right\}  \text{ for
some }N\in\mathbb{N}.
\]
By applying (\ref{fund_sol_par}) we see that both integrals tend to zero as
$t\rightarrow0^{+}$.
\end{proof}

\subsection{The uniqueness of the solution of the Cauchy problem}

The technique used in \cite[Theorem 4.5]{Koch} for establishing the uniqueness
of the solution of the Cauchy problem associated with perturbations of the
Vladimirov operator can be used to show the uniqueness of the Cauchy problem
(\ref{cauchy_Pr_2}). Then we state here the corresponding result without proof.

\begin{theorem}
\label{Thm3}Assume that the coefficients $a_{k}(x,t)$, $k=0,\ldots,n$ are
non-negative, bounded, continuous functions and that $b(x,t)$ is bounded,
continuous function. Take $0\leq\lambda<\alpha_{1}$(if $a_{1}(x,t)=\dots
=a_{n}(x,t)\equiv0$, then we shall assume that $0\leq\lambda<\alpha$). If
$u\left(  x,t\right)  $ is a solution of (\ref{cauchy_Pr_2}) with $g\left(
x,t\right)  \equiv0$, such that $u\in\mathfrak{M}_{2\lambda}$ uniformly with
respect to $t$, and $\varphi\left(  x\right)  \equiv0$, then $u\left(
x,t\right)  \equiv0$.
\end{theorem}

\subsection{Probabilistic interpretation}

The fundamental solution of (\ref{cauchy_Pr_2}) $\Gamma(x,t,\xi,\tau)$ is the
transition function for a random walk on $\mathbb{Q}_{p}^{4}$. This result can
be established by using classical results on stochastic processes see e.g.
\cite{Dyn} and the techniques given in \cite[pp. 161-162]{Koch}. Formally we have

\begin{theorem}
\label{Thm4}Assume that the coefficients $a_{k}(x,t)$, $k=0,\ldots,n$ and
$b(x,t)$ are non-negative, bounded, continuous functions. The fundamental
solution $\Gamma(x,t,\xi,\tau)$ is the transition density of a bounded
right-continuous Markov process without second kind discontinuities.
\end{theorem}

\subsection{Quadratic forms of dimension two}

All the results presented in this article are valid if for the quadratic forms
of type $\xi_{1}^{2}-\tau\xi_{2}^{2}$ where $\tau\in\mathbb{Q}_{p}%
\smallsetminus\left\{  0\right\}  $ is a not square of an element of
$\mathbb{Q}_{p}$, see \cite{C-Zu1}.

\mathstrut

\end{document}